\documentclass[oneside,12pt]{amsart}

\usepackage{amssymb}
\usepackage{amsmath}
\usepackage[all]{xy}

\usepackage{hyperref}

\usepackage[russian,english]{babel}

\def\r0{{A^1}}

\DeclareMathOperator{\Frac}{Frac}
\DeclareMathOperator{\der}{der}
\DeclareMathOperator{\corad}{corad}

\newcommand{\incl}[1][r]
{\ar@<-0.2pc>@{^(-}[#1] \ar@<+0.2pc>@{-}[#1]}

\newtheorem{lem}{Lemma}[section]
\newtheorem*{thm*}{Theorem}
\newtheorem{thm}[lem]{Theorem}

\newtheorem{cor}[lem]{Corollary}
\theoremstyle{definition}{    }
\theoremstyle{definition}{    }
\theoremstyle{definition}{  \newtheorem{conj}[lem]{Conjecture} }

\newcounter{nc}
\renewcommand{\thenc}{{\rm(\roman{nc})}}
\newenvironment{romlist}%
{\begin{list}{\thenc}{
\usecounter{nc}
\parsep=0pt
\setlength  \labelwidth{\leftmargin}
\addtolength\labelwidth{-\labelsep}
}
}{\end{list}}
\newcounter{nnc}
\renewcommand{\thennc}{{\rm(\alph{nnc})}}
{\begin{list}{\thennc}{
\usecounter{nnc}
\parsep=0pt
\setlength  \labelwidth{\leftmargin}
\addtolength\labelwidth{-\labelsep}
}
}{\end{list}}
\newcommand{\pauseromlist}%
{\global\edef\savecount{\arabic{nc}}\end{romlist}}
\newcommand{\finpauseromlist}%
{\begin{romlist}\setcounter{nc}{\savecount}}

\newcounter{ctnum}
\renewcommand{\thectnum}{\textup{(\arabic{ctnum})}}
{\begin{list}{\thectnum}{
\usecounter{ctnum}
\parsep=0pt
\leftmargin=0pt%
\setlength{\itemindent}{\labelwidth}%
\addtolength{\itemindent}{\labelsep}%
}
}{\end{list}}

\def\fppf{\text{\rm fppf}}

\def\et{\text{\rm \'et}}

\def\2int{\mathop{2\int}\nolimits}

\def\dim{\mathop{\rm dim}\nolimits}

\def\Spec{\mathop{\rm Spec}\nolimits}

\def\resp.{\mathop{\rm resp.}\nolimits}

\font\math=cmmi10
\def\varpi{\hbox{\math\char'44}}

\def\pa{\S\kern.15em }

\def\un{\uppercase\expandafter{\romannumeral 1}}
\def\deux{\uppercase\expandafter{\romannumeral 2}}
\def\trois{\uppercase\expandafter{\romannumeral 3}}
\def\quatre{\uppercase\expandafter{\romannumeral 4}}
\def\cinq{\uppercase\expandafter{\romannumeral 5}}
\def\six{\uppercase\expandafter{\romannumeral 6}}

\def\et{\acute et}

\def\hfl#1#2#3{\smash{\mathop{\hbox to#3{\rightarrowfill}}\limits
^{\scriptstyle#1}_{\scriptstyle#2}}}
\def\gfl#1#2#3{\smash{\mathop{\hbox to#3{\leftarrowfill}}\limits
^{\scriptstyle#1}_{\scriptstyle#2}}}

\title{On the generalized Bass--Quillen conjecture\\ in dimension 2}

\author{A. Stavrova}\address{
St. Petersburg Department of Steklov Mathematical Institute, nab. r. Fontanki 27, 191023 St. Petersburg, Russia
}

\email{anastasia.stavrova@gmail.com}

\date{\today}

\begin{document}

\maketitle

\begin{abstract}
Let $A$ be a regular ring of dimension $\le 2$.
Let $G$ be a
reductive group over $A$ such that its derived group is a split, i.e. a Chevalley--Demazure,
semisimple group. We prove that every Zariski-locally trivial principal $G$-bundle over
$A[x_1,\ldots,x_n]$ is extended from $A$, for any $n\ge 1$. This result generalizes to split reductive groups
the dimension $2$ case of the Bass--Quillen conjecture on finitely generated projective modules, settled
in positive by M. P. Murthy.
\end{abstract}

\section{Introduction}

The following conjecture was initially stated by H. Bass as a problem~\cite[Problem IX]{Bass-problems},
and then popularized as the Bass--Quillen conjecture after it was rehearsed by D. Quillen in~\cite{Q}; see e.g.~\cite{Lam}
for a detailed historical account.

\begin{conj}(Bass--Quillen conjecture)
Let $A$ be a regular ring of finite Krull dimension. Then every finitely generated projective module over $A[x_1,\ldots,x_n]$
is extended from $A$ for any $n\ge 1$.
\end{conj}

This conjecture can be viewed as a specialization to $G=GL_m$, $m\ge 1$, of the following more general statement.
Let $A$ be a regular ring, and let $G$ be a strictly isotropic reductive group over $A$, then every Zariski-locally
trivial principal $G$-bundle over $A[x_1,\ldots,x_n]$ is extended from $A$ for any $n\ge 1$.
Here a reductive group (scheme)
$G$ in the sense of~\cite{SGA3} is called strictly isotropic, if $G$ contains a parabolic subgroup
that intersects properly every semisimple normal subgroup of $G$, see e.g.~\cite{St-k1}
(the same property is sometimes called ``totally isotropic''). This generalized version of the Bass--Quillen conjecture
is known for equicharacteristic (i.e. containing a field) regular rings $A$ of arbitrary dimension
and every strictly isotropic reductive group $G$ defined over $A$~\cite[Corollary 5.5]{St-k1} (see also~\cite{Ces-bqnis}
for a different proof). Also, it is known for all rings $A$ which are ind-smooth over a discrete valuation ring $D$
and for all strictly isotropic reductive groups $G$ defined over $D$~\cite[Theorem 1.8]{NGFL-bq}.
See also~\cite{Lam,Ces-Prob,Sus-K-O1,K91,CTO,AHW,RaRa,Rag} for earlier results and discussion of the assumptions.

We establish the dimension $\le 2$ case of the generalized Bass--Quillen conjecture
for all reductive groups $G$ whose derived group $G^{der}$ in the sense of~\cite{SGA3} is a split, i.e. Chevalley--Demazure,
 semisimple group. Such groups $G$ are automatically strictly isotropic, since all of their semisimple normal subgroups
are split and contain a Borel subgroup.

\begin{thm}\label{thm:main}
Let $A$ be a regular ring of dimension $\le 2$.
Let $G$ be a
reductive group over $A$ such that $G^{der}$ is a split, i.e. Chevalley--Demazure,
 semisimple group.
Then the map
$$
H^1_{Zar}(A[x_1,\ldots,x_n],G)\xrightarrow{x_1=\ldots=x_n=0} H^1_{Zar}(A,G)
$$
is bijective for any $n\ge 1$.
That is, every Zariski-locally trivial principal $G$-bundle over
$A[x_1,\ldots,x_n]$ is extended from $A$.
\end{thm}

If $G=GL_n$, then Theorem~\ref{thm:main} is due to M. P. Murthy~\cite{Mur66} for $n=1$ and to
A. Suslin~\cite[Theorem 5]{Su76} and D. Quillen~\cite{Q} in general. Raman Parimala established the analogous statement
for isotropic quadratic spaces over $A[x_1,\ldots,x_n]$ with $2\in A^\times$~\cite{Par82}.

\section{Preliminaries}

For any commutative ring $R$, we denote by $R(x)$ the localization of the polynomial ring $R[x]$ at the set of all monic
polynomials.

Let $G$ be an affine locally finitely presented group scheme over $R$.
We denote by $H^1_{Zar}(R,G)$, $H^1_{\et}(R,G)$ and $H^1_{fppf}(R,G)$ the first non-abelian cohomology of
$\Spec(R)$ with values in $G$ with respect to Zariski, \'etale and fppf topologies, respectively. Equivalently, these are the sets of
isomorphism classes of principal $G$-bundles (or $G$-torsors) over $\Spec(R)$ which are trivial locally in the corresponding
topology. We write $\xi=1$ to signify that the principal $G$-bundle $\xi$ is trivial.
For brevity, we also denote
$$
H^1_{\et}(R[x_1,\ldots,x_n],G)_0=\ker\left(H^1_{\et}(R[x_1,\ldots,x_n],G)\xrightarrow{x_1=\ldots=x_n=0} H^1_{\et}(R,G)\right),
$$
and similarly for the fppf and Zariski topologies.

The following lemma is a slightly stronger version of~\cite[Lemma 4.1]{St-dh}.

\begin{lem}\label{lem:TG'GT}
Let $R$ be a regular semilocal domain, $K$ be the fraction field of $R$. Let $G,G'$ be reductive $R$-groups, and $T$ be an
$R$-group of multiplicative type such that there is a short exact sequence
$$
(a)\quad 1\to G'\to G\to T\to 1\quad\mbox{or}\quad (b)\quad 1\to T\to G'\to G\to 1
$$
of $R$-group schemes.
Fix any $n\ge 1$, and write $\mathbf{x}$ for the string of variables $x_1,\ldots,x_n$.
The natural map
$H^1_{\et}(R[\mathbf{x}],G')\to H^1_{\et}(R[\mathbf{x}],G)$
induces an isomorphism
\begin{equation}\label{eq:GG'}
\ker\left(H^1_{\et}(R[\mathbf{x}],G')_0\to H^1_{\et}(K[\mathbf{x}],G')\right)
\xrightarrow{\cong}
\ker\left(H^1_{\et}(R[\mathbf{x}],G)_0\to H^1_{\et}(K[\mathbf{x}],G)\right).
\end{equation}
\end{lem}
\begin{proof}
Since $G$ and $G'$ are smooth, we can replace their \'etale cohomology by fppf, see e.g.~\cite[p. 120]{Milne-etco}.

Consider first the case $(a)$. Let $S$ be any of $R$, $R[\mathbf{x}]$, $K$, $K[\mathbf{x}]$, then we have
an exact sequence of pointed sets
$$
T(S)\xrightarrow{\delta} H^1_{\fppf}(S,G')\to
H^1_{\fppf}(S,G)\to H^1_{\fppf}(S,T).
$$
Since $R$ is regular, $T$ is isotrivial~\cite[Exp. X, Th. 5.16]{SGA3}, and hence subject to the results of~\cite{CTS}. Namely,
by~\cite[Lemma 2.4]{CTS} we have $H^1_{\fppf}(S,T)\cong H^1_{\fppf}(S[x],T)$,
and by~\cite[Theorem 4.1]{CTS} the map $H^1_{\fppf}(R,T)\to H^1_{\fppf}(K,T)$ is injective.
Hence $H^1_{\fppf}(R[\mathbf{x}],T)\to H^1_{\fppf}(K[\mathbf{x}],T)$ is also injective.
Hence any
$\xi\in\ker\bigl(H^1_{\fppf}(R[\mathbf{x}],G)\to H^1_{\fppf}(K[\mathbf{x}],G)\bigr)$ lifts to an element $\eta\in H^1_{\fppf}(R[\mathbf{x}],G')$.
Let $\bar\eta\in H^1_{\fppf}(K[\mathbf{x}],G')$ be the image of $\eta$. Then there is $\theta\in T(K[\mathbf{x}])$ such that
$\bar\eta$ is the image of $\theta$. Since $T$ is a group of multiplicative type, we have $T(K[\mathbf{x}])=T(K)$. Hence
$\bar\eta$ is extended from $K$.

Assume that, moreover, $\xi\in \ker\bigl(H^1_{\fppf}(R[\mathbf{x}],G)_0\to H^1_{\fppf}(K[\mathbf{x}],G)\bigr)$.
Then $\eta|_{\mathbf{x}=(0,\ldots,0)}\in H^1_{fppf}(R,G')$
has a preimage $\sigma\in T(R)$. Clearly, the image of $\sigma$ in $T(K)$ maps to
$\bar\eta|_{\mathbf{x}=(0,\ldots,0)}\in H^1_{fppf}(K,G')$ under $\delta$.

The group $T(R[\mathbf{x}])$ acts on
$H^1_{\fppf}(R[\mathbf{x}],G')$ by right shifts, and the orbits are precisely the fibers of the map
$
H^1_{\fppf}(R[\mathbf{x}],G')\to H^1_{\fppf}(R[\mathbf{x}],G),
$
see~\cite[\S 5.5]{Serre-gal}.
Since $R$ is a domain, we have $T(R[\mathbf{x}])=T(R)$.
Then $\eta\cdot\sigma^{-1}\in H^1_{\fppf}(R[\mathbf{x}],G')_0$, and,
since $\bar\eta$
is extended from $K$, the image of $\eta\cdot\sigma^{-1}$ in $H^1_{\fppf}(K[\mathbf{x}],G')$ is trivial.
Hence
$$
\eta\cdot\sigma^{-1}\in
\ker\bigl(H^1_{\fppf}(R[\mathbf{x}],G')_0\to H^1_{\fppf}(K[\mathbf{x}],G')\bigr).
$$
Since the image of $\eta\cdot\sigma^{-1}$ in $H^1_{\fppf}(R[\mathbf{x}],G)$ coincides with $\xi$, the surjectivity
of the map~\eqref{eq:GG'} is proved. The injectivity of~\eqref{eq:GG'} follows immediately from the fact that its fibers
are orbits under $T(R[x])=T(R)$.

Consider the case $(b)$. For each $S$ as above, we have
an exact sequence
$$
H^1_{\fppf}(S,T)\to H^1_{\fppf}(S,G')\to
H^1_{\fppf}(S,G)\to H^2_{\fppf}(S,T).
$$
By~\cite[Theorem 4.3]{CTS} the map
$
H^2_{\fppf}(R[\mathbf{x}],T)\to H^2_{\fppf}(K(\mathbf{x}),T)$
is injective, hence the map
$H^2_{\fppf}(R[\mathbf{x}],T)\to H^2_{\fppf}(K[\mathbf{x}],T)$ is also injective.
The rest of the proof is the same as in the previous case, with the only difference that one uses the action
of the commutative group $H^1_{\fppf}(R[\mathbf{x}],T)\cong H^1_{\fppf}(R,T)$ on
$H^1_{\fppf}(R[\mathbf{x}],G')$, which is well-defined since $T$ is central in $G'$, see~\cite[\S 5.7]{Serre-gal}.
\end{proof}

\begin{cor}\label{cor:GGsc}
Let $R$ be a regular semilocal domain, $K$ be the fraction field of $R$. Let $G$ be a reductive group over $R$.
Let $G^{sc}$ be the simply connected cover of the derived group of $G$ in the sense of~\cite{SGA3}.
Fix any $n\ge 1$, and write $\mathbf{x}$ for the string of variables $x_1,\ldots,x_n$.
The natural homomorphism $G^{sc}\to G$ induces an isomorphism
\begin{equation}
\ker\left(H^1_{\et}(R[\mathbf{x}],G^{sc})_0\to H^1_{\et}(K[\mathbf{x}],G^{sc})\right)
\xrightarrow{\cong}
\ker\left(H^1_{\et}(R[\mathbf{x}],G)_0\to H^1_{\et}(K[\mathbf{x}],G)\right).
\end{equation}
\end{cor}
\begin{proof}
Let
$\der(G)$  be the derived group of $G$ in the sense of~\cite{SGA3}.
There are two short exact sequences of $R$-groups $1\to\der(G)\to G\to\corad(G)\to 1$ and
$1\to C\to G^{sc}\to \der(G)\to 1$,
where $\corad(G)$ and $C$ are $R$-groups of multiplicative type~\cite[Exp. XXII]{SGA3}.
These two short exact sequences are subject to
Lemma~\ref{lem:TG'GT}.
\end{proof}

The proof of our main results crucially depends on the following theorem.

\begin{thm}\label{thm:D}\cite{Har67,Ni,Guo-ded}
Let $D$ be a Dedekind domain and let $K$ be its fraction field. Let $G$ be a simply connected quasi-split reductive group
over $D$. Then
$$
H^1_{Zar}(D,G)=\ker(H^1_{\et}(D,G)\to H^1_{\et}(K,G))=1.
$$
\end{thm}
\begin{proof}
One has  $\ker(H^1_{\et}(D,G)\to H^1_{\et}(K,G))\subseteq
H^1_{Zar}(D,G)$ by the Serre--Grothendieck conjecture for discrete valuation rings~\cite{Ni,Guo-ded}.
By~\cite[Satz 3.3]{Har67} for quasi-split simply connected groups one has $H^1_{Zar}(D,G)=1$ (see~\cite{PaSt-Har}
for a more general statement).
\end{proof}

For any commutative ring $R$ with 1 and any fixed choice of a pinning, or \'epinglage of a Chevalley--Demazure group
$G$ in the sense of~\cite{SGA3},
we denote by $E(R)$ the elementary subgroup of $G(R)$.
That is, $E(R)$ is the subgroup of $G(R)$ generated
by elementary root unipotent elements $x_\alpha(r)$, $\alpha\in\Phi$, $r\in R$, in the notation of~\cite{Che55,Ma},
where $\Phi$ is the root system of $G$. (A generalization of this notion to isotropic reductive groups was
introduced in~\cite{PS}.)

We will need the following lemma on the elementary subgroup which follows from
the stability theorems of M. R. Stein and E. B. Plotkin~\cite{Ste78,Plo93}.

\begin{lem}\label{lem:stability}\cite[Lemma 3.1]{St-ded}
Let $R$ be a Noetherian ring of Krull dimension $\le 1$. If $SL_2(R)=E_2(R)$, then
$G(R)=E(R)$ for any simply connected Chevalley--Demazure group scheme $G$.
\end{lem}

\section{The case of one variable}

We deduce the $n=1$ case of Theorem~\ref{thm:main} from the following result about \'etale-locally trivial
principal $G$-bundles.

\begin{thm}\label{th:BQdim2split}
Let $A$ be a 2-dimensional regular local ring with the fraction field $K$. Let $G$ be a
reductive group over $A$ such that its derived group $G^{der}$ in the sense of~\cite{SGA3} is a split, i.e. Chevalley--Demazure,
 semisimple group.
Then
$$
\ker\left(H^1_{\et}(A[x],G)_0\to H^1_{\et}(K[x],G)\right)=
\ker\left(H^1_{\et}(A[x],G)_0\to H^1_{\et}(K(x),G)\right)=1.
$$
\end{thm}
\begin{proof}
The equality between kernels holds since
$$
\ker\left(H^1_{\et}(K[x],G)\to H^1_{\et}(K(x),G)\right)=1
$$
by~\cite[Prop. 2.2]{CTO}.
By Corollary~\ref{cor:GGsc} it is enough to prove the claim for the simply connected cover $G^{sc}$ of $G^{der}$.
Since $G^{der}$ is split, it follows that $G^{sc}$ is split as well. So, from now on, we assume that $G=G^{sc}$
is a simply connected split semisimple group.

The kernel of the map $H^1_{\et}(A[x],G)\to H^1_{\et}(A(x),G)$ is trivial~\cite[Theorem 4.1]{St-k1} (see
also~\cite[Theorem 1.3]{PaStV},~\cite[Prop. 8.4]{C1}, and~\cite{PaSt-gth} for other versions of this result).
Thus, it is enough to show that for every
$$
\xi\in\ker\left(H^1_{\et}(A[x],G)_0\to H^1_{\et}(K[x],G)\right)
$$
the image of $\xi$ in $H^1_{\et}(A(x),G)$ is trivial.

Let $s$ be a regular parameter of $A$.
 We have the commutative diagram
\begin{equation*}
\xymatrix@R=15pt@C=20pt{
A[x]\ar[r]\ar[d]&A(x)\ar[d]\\
A_s[x]\ar[r]& A(x)_s\ar[r] & A_s(x)\ar[r] & K(x).\\
}
\end{equation*}
The kernel of $H^1_{\et}(A_s[x],G)\to H^1_{\et}(A_s(x),G)$ is trivial (again, by~\cite[Theorem 4.1]{St-k1}).
Since $A_s$ is a Dedekind domain, $A_s(x)$ is also a Dedekind domain~\cite[Ch. IV, Cor. 1.3]{Lam}, hence
the kernel of $H^1_{\et}(A_s(x),G)\to H^1_{\et}(K(x),G)$ is trivial by Theorem~\ref{thm:D}.
Then the kernel of $H^1_{\et}(A_s[x],G)\to H^1_{\et}(K(x),G)$ is trivial.
Hence the image of $\xi$ in $H^1_{\et}(A(x)_s,G)$ is trivial.

Consider the Cartesian square
\begin{equation}\label{eq:squareA(x)}
\xymatrix@R=15pt@C=20pt{
A(x)\ar[r]\ar[d]&A(x)_s\ar[d]\\
\widehat{A(x)}\ar[r]& \widehat{A(x)}_s,\\
}
\end{equation}
where $\widehat{A(x)}$ denotes the $s$-adic completion of $A(x)$, and $\widehat{A(x)}_s$ is its localization at $s$.
This is a patching square for $G$-torsors~\cite[Lemma 2.2.11]{BC}.

We claim that the restriction of $\xi$ to $\widehat{A(x)}$ is trivial.
Indeed, since $\bigl(\widehat{A(x)},\, s\widehat{A(x)}\bigr)$ is a henselian pair, by~\cite[Theorem 2.1.6]{BC} we have
\begin{equation}\label{eq:hpair}
H^1_{\et}(\widehat{A(x)},G)\cong H^1_{\et}(\widehat{A(x)}/s,G).
\end{equation}
Furthermore, we have $\widehat{A(x)}/s\cong (A/s)(x)$ is a regular domain of dimension $\le 1$, and its fraction field
satisfies
$$
\Frac\left((A/s)(x)\right)\cong\Frac(A/s)(x)\cong A_{(s)}(x)/s,
$$
where $A_{(s)}$ is the localization of $A$ at the prime ideal $(s)$. Then  by Theorem~\ref{thm:D}
$$
\ker\left(H^1_{\et}((A/s)(x),G)\to H^1_{\et}(A_{(s)}(x)/s,G)\right)=1.
$$
Since $A_{(s)}(x)$ is itself a regular domain of dimension $\le 1$, again by Theorem~\ref{thm:D}
we have
$$
\ker\left(H^1_{\et}(A_{(s)}(x),G)\to H^1_{\et}(K(x),G)\right)=1.
$$
Since $\xi$ is trivial over $K(x)$, it follows that $\xi$ is trivial over $A_{(s)}(x)$, hence over $A_{(s)}(x)/s$,
hence over $(A/s)(x)$. Then by~\eqref{eq:hpair} the restriction of $\xi$ to $\widehat{A(x)}$ is also trivial.

Thus, $\xi|_{A(x)}$ corresponds to a $G$-torsor whose restrictions to the corners $\widehat{A(x)}$ and $A(x)_s$
of the square~\eqref{eq:squareA(x)} are trivial.
By Lemma~\ref{lem:hatAx} below we have $G(\widehat{A(x)}_s)=E(\widehat{A(x)}_s)$. By~\cite[Lemma 3.2]{PaSt}
we have
$$
E(\widehat{A(x)}_s)\subseteq G(\widehat{A(x)})\cdot E(A(x)_s).
$$
Consequently,
$G(\widehat{A(x)}_s)\subseteq G(\widehat{A(x)})\cdot G(A(x)_s)$.
Then $\xi$ is trivial over $A(x)$.
\end{proof}

\begin{lem}\label{lem:hatAx}
Let $A$ be a regular local ring of dimension $\le 2$, and let $s\in A$ be a regular parameter. Let $\widehat{A(x)}$ be the
$s$-adic completion of $A$. Then the localization $D=\widehat{A(x)}_s$ of $\widehat{A(x)}$ at $s$
is a regular domain of dimension $\le 1$,
and $G(D)=E(D)$ for any simply connected Chevalley--Demazure group scheme $G$.
\end{lem}
\begin{proof}
We have
$$
\widehat{A(x)}/s\cong A(x)/s\cong (A/s)(x).
$$
If $\dim A=1$, then $(A/s)(x)$ is a field, and hence
$SL_n((A/s)(x))=E_n((A/s)(x))$ for any $n\ge 2$.
If $\dim A=2$, then $A/s$ is a discrete valuation ring, and
$(A/s)(x)$ is a special principal ideal domain in the sense of~\cite[Ch. IV, \S~6]{Lam}. Consequently,
$SL_n((A/s)(x))=E_n((A/s)(x))$ for any $n\ge 2$ by~\cite[Ch. IV, Corollary 6.3]{Lam}.

Since $\bigl(\widehat{A(x)},\, s\widehat{A(x)}\bigr)$ is a henselian pair, by~\cite[Prop. 7.7]{GSt} we have
$$
SL_n(\widehat{A(x)})/E_n(\widehat{A(x)})\cong SL_n((A/s)(x))/E_n((A/s)(x))=1
$$
for any $n\ge 2$.
Since, again, $(A/s)(x)$ is a principal ideal domain, by~\cite[Ch. IV, Lemma 6.1]{Lam} we have that, for every $n\ge 2$,
$SL_n(D)=E_n(D)$, where $D=\widehat{A(x)}_s$.

Note that $\dim(D)\le 1$. Indeed, $\dim(\widehat{A(x)})\le 2$
and all maximal ideals in $\widehat{A(x)}$ contain $s$.
It follows that all maximal
ideals in  $D$ have height $\le 1$.
Then by Lemma~\ref{lem:stability} we have $G(D)=E(D)$
for any simply connected Chevalley--Demazure group scheme $G$.
\end{proof}

\begin{cor}\label{cor:A[x]}
Let $A$ be a regular ring of dimension $\le 2$.
 Let $G$ be a
split (i.e. a Chevalley--Demazure) reductive group over $\mathbb{Z}$.
Then $H^1_{Zar}(A[x],G)=H^1_{Zar}(A,G)$.
In other words, every Zariski-locally trivial principal $G$-bundle over
$A[x]$ is extended from $A$.
\end{cor}
\begin{proof}
The group $G$ is linear by~\cite[Corollary 3.2]{Thomason}, since it is defined over the regular ring $\mathbb{Z}$.
Then by the generalized Quillen's local-global principle~\cite[Theorem 3.2.5]{AHW} (see also~\cite{Moser}), to prove that $\xi\in H^1_{Zar}(A[x],G)$ is extended from $A$, it is enough to prove that
for every maximal ideal $m$ of $A$ $\xi|_{A_m[x]}$ is trivial. The latter holds by Theorem~\ref{th:BQdim2split},
since $(\xi|_{A_m[x]})|_{x=0}$ is a Zariski-locally trivial $G$-torsor over $A_m$, hence trivial.
\end{proof}

\section{The $n$-variable case}

\begin{lem}\label{lem:sc-Dxn}
Let $A$ be a regular ring of dimension $\le 2$ such that the fraction field of
every maximal localization of $A$ has characteristic $0$.
 Let $G$ be a
split (i.e. a Chevalley--Demazure) simply connected semisimple group over $\mathbb{Z}$.
Then for any $n\ge 1$
$$
H^1_{\et}(A[x_1,\ldots,x_n],G)_0=1.
$$
\end{lem}
\begin{proof}
We proceed by induction on $n$.

The group $G$ is linear by~\cite[Corollary 3.2]{Thomason}, since it is defined over the regular ring $\mathbb{Z}$.
Then by~\cite[Theorem 3.2.5]{AHW} a torsor
$\xi\in H^1_{\et}(A[x_1,\ldots,x_n],G)$ is extended from $A$ if and only if for any maximal ideal $m\subseteq A$
the image of $\xi$ in $H^1_{\et}(A_m[x_1,\ldots,x_n],G)$ is extended from $A_m$. Thus, to prove the claim,
it is enough to show that it holds for every regular local ring $A$ such that $\dim(A)\le 2$ and the fraction field $K$ of $A$
has characteristic $0$.

Assume $n=1$. By Theorem~\ref{th:BQdim2split} the map
$$
H^1_{\et}(A[x_1],G)_0\to H^1_{\et}(K[x_1],G)
$$
has trivial kernel. Since $K$ has characteristic $0$, by the main result of~\cite{RaRa} the map
$$
H^1_{\et}(K[x_1],G)\xrightarrow{x_1=0} H^1_{\et}(K,G)
$$ is bijective. It follows that $H^1_{\et}(A[x_1],G)_0=1$.

Assume $n\ge 2$. The ring $A(x_1)$ is a regular domain such that
$\dim(A(x_1))\le 2$~\cite[Ch. IV, Cor. 1.3]{Lam}.
Clearly, the fraction field $K(x_1)$ of $A(x_1)$ has characteristic $0$. Then by the inductive assumption
$$
H^1_{\et}(A(x_1)[x_2,\ldots,x_n],G)_0=1.
$$
Consider the commutative diagram
\begin{equation*}
\xymatrix@R=15pt@C=70pt{
H^1_{\et}(A[x_1][x_2,\ldots,x_n],G)\ar[d]\ar[r]^{\quad x_2=\ldots=x_n=0}&H^1_{\et}(A[x_1],G)\ar[r]^{x_1=0}\ar[d] & H^1_{\et}(A,G)\\
H^1_{\et}(A(x_1)[x_2,\ldots,x_n],G)\ar[r]^{\quad x_2=\ldots=x_n=0}&H^1_{\et}(A(x_1),G)&\\
}
\end{equation*}
It implies that if $\xi\in H^1_{\et}(A[x_1,\ldots,x_n],G)_0$, then $\xi$ is in the kernel of
$$
H^1_{\et}(A[x_1,\ldots,x_n],G)\to H^1_{\et}(A(x_1)[x_2,\ldots,x_n],G).
$$
Set $B=A[x_2,\ldots,x_n]$.
Thus there is a monic polynomial $f(x_1)\in B[x_1]$ such that
$$
\xi\in\ker\left(H^1_{\et}(B[x_1],G)\to H^1_{\et}(B[x_1]_f,G)\right).
$$
By~\cite[Theorem 4.1]{St-k1}
the latter map has trivial kernel. Hence $\xi$ is trivial.
\end{proof}

\begin{lem}\label{lem:red-DK}
Let $A$ be a semilocal regular domain of dimension $\le 2$
such that its fraction field $K$ has characteristic $0$.
Let $G$ be a
reductive group over $A$ such that its derived group $G^{der}$ in the sense of~\cite{SGA3} is a split, i.e. Chevalley--Demazure,
 semisimple group.
Then
$$
\ker\bigl(H^1_{\et}(A[x_1,\ldots,x_n],G)_0\to H^1_{\et}(K[x_1,\ldots,x_n],G)\bigr)=1.
$$
\end{lem}
\begin{proof}
Follows from Lemma~\ref{lem:sc-Dxn} and Corollary~\ref{cor:GGsc}.
\end{proof}

\begin{proof}[Proof of Theorem~\ref{thm:main}]
The group $G$ is $A$-linear by~\cite[Corollary 3.2]{Thomason}. Then by~\cite[Theorem 3.2.5]{AHW} a torsor
$\xi\in H^1_{\et}(A[x_1,\ldots,x_n],G)$ is extended from $A$ if and only if for any maximal ideal $m\subseteq A$
the image of $\xi$ in $H^1_{\et}(A_m[x_1,\ldots,x_n],G)$ is extended from $A_m$.

Thus, we need to show that $H^1_{Zar}(A[x_1,\ldots,x_n],G)_0=1$.
Let $B=A_m$ be a maximal localization of $A$. Then $B$ is a regular local ring of dimension $\le 2$.
If $B$ contains a field, then every Zariski-locally trivial principal $G$-bundle over $B[x_1,\ldots,x_n]$
is extended from $B$ by~\cite[Corollary 5.5]{St-k1}. Otherwise the fraction field of $B$ has characteristic 0.
Then
by Lemma~\ref{lem:red-DK} we have
$$
\ker\bigl(H^1_{\et}(A[x_1,\ldots,x_n],G)_0\to H^1_{\et}(K[x_1,\ldots,x_n],G)\bigr)=1.
$$
 Again by~\cite[Corollary 5.5]{St-k1}, the map
$$
H^1_{Zar}(K[x_1,\ldots,x_n],G)\xrightarrow{x_1=\ldots=x_n=0} H^1_{Zar}(K,G)
$$
is bijective, and hence
$$
H^1_{Zar}(B[x_1,\ldots,x_n],G)_0\subseteq
\ker\bigl(H^1_{\et}(A[x_1,\ldots,x_n],G)_0\to H^1_{\et}(K[x_1,\ldots,x_n],G)\bigr)=1,
$$
as required.
\end{proof}

\bigskip

\medskip


\begin{thebibliography}{MMM}


\bibitem[AHW]{AHW}
 A. Asok, M. Hoyois, M. Wendt, {\it Affine representability results in
 $\mathbf{A}^1$-homotopy theory, II: Principal bundles and homogeneous spaces},
 Geom. Topol. {\bf 22} (2018),  1181--1225.



\bibitem[Ba]{Bass-problems} H. Bass, {\it
Some problems in ``classical'' algebraic K-theory}, in Algebraic K-Theory, II: ``Classical''
Algebraic K-Theory and Connections with Arithmetic (Proceedings of the Conference,
Battelle Memorial Institute, Seattle, Washington, 1972). Lecture Notes in Mathematics, vol.
342 (1973), pp. 3--73.






\bibitem[BC]{BC} A.~Bouthier, K.~\v{C}esnavi\v{c}ius,
{\it Torsors on loop groups and the Hitchin fibration},
 Annales de l'Ecole normale sup\'erieure {\bf  55} (2022), 791--864.





\bibitem[C1]{C1}
K. \v{C}esnavi\v{c}ius, \emph{Grothendieck--Serre in the quasi-split unramified case},
Forum of Mathematics, Pi (2022), \textbf{10:e9}, 1--30.

\bibitem[C2]{Ces-Prob} K.~\v{C}esnavi\v{c}ius, {\it Problems about torsors over regular rings},
Acta Math. Vietnam. {\bf 47} (2022), no. 1, 39--107.


\bibitem[C3]{Ces-bqnis}
K. \v{C}esnavi\v{c}ius, \emph{Torsors on the complement of a smooth divisor}, Cambridge J. Math. \textbf{12} (2024), 903--936.





\bibitem[Ch]{Che55}
C.~Chevalley, \emph{Sur certains groupes simples}, Tohoku Math. J. \textbf{7}
  (1955), 14--66.


\bibitem[CTO]{CTO} J.--L. Colliot--Th\'el\`ene, M. Ojanguren,
 {\it Espaces principaux homog\`enes localement triviaux},
Publications Math\'ematiques de l'IHES {\bf 75} (1992),  97-122.






\bibitem[CTS]{CTS} J.-L. Colliot-Th\'el\`ene and J.-J. Sansuc, {\it
Principal homogeneous spaces under flasque tori: applications}, J.
Algebra {\bf 106} (1987),  148--205.




\bibitem[SGA3]{SGA3} M.~Demazure, A.~Grothendieck, \emph{Sch\'emas en groupes}, Lecture Notes in
Mathematics, vol. 151--153, Springer-Verlag, Berlin-Heidelberg-New York, 1970.


\bibitem[GiSt]{GSt} P. Gille, A. Stavrova, \emph{$R$-equivalence on reductive group schemes}, to appear in Algebra
\& Number Theory.









\bibitem[H67]{Har67} G. Harder, {\it
Halbeinfache Gruppenschemata \"uber Dedekindringen}, Invent. Math. 4 (1967), 165--191.




\bibitem[K]{K91} V. I. Kopeiko, \emph{On a theorem of Ojanguren},
J. Soviet Math., {\bf 63} (1993), 683.




\bibitem[La]{Lam} T. Y. Lam, {\it  Serre's problem on projective modules}
  Springer Monographs in Mathematics, Springer, Berlin, Heidelberg, New York, 2006.





\bibitem[Ma]{Ma} H.~Matsumoto, \emph{Sur les sous-groupes
arithm\'etiques des groupes semi-simples d\'eploy\'es},
Ann. Sci.  de l'\'E.N.S. $4^e$ s\'erie, tome 2, n. 1 (1969), 1--62.


\bibitem[Mi]{Milne-etco}
J.~S. Milne, \emph{\'{E}tale cohomology}, Princeton Mathematical Series,
  vol.~33, Princeton University Press, Princeton, N.J., 1980.


\bibitem[Mo]{Moser} L-F. Moser, {\it Rational triviale Torseure und die
Serre-Grothendiecksche Vermutung}, Diplomarbeit, 2018.

\bibitem[Mur66]{Mur66}
M.~Pavaman Murthy, \emph{Projective {$A[X]$}-modules}, J. London Math. Soc.
  \textbf{41} (1966), 453--456.



\bibitem[Gu]{Guo-ded} Ning Guo, \emph{The Grothendieck--Serre conjecture over semilocal Dedekind rings},
Transformation Groups \textbf{27} (2022), 897--917.


\bibitem[GuLi]{NGFL-bq} Ning Guo, Fei Liu, {\it The Bass--Quillen conjecture for torsors over valuation rings},
IMRN \textbf{11} (2025), 1--17.

\bibitem[Ni84]{Ni} Y. Nisnevich,
\emph{Rationally Trivial Principal Homogeneous Spaces and Arithmetic of
Reductive Group Schemes Over Dedekind Rings},
C. R. Acad. Sc. Paris, S\'erie I {\bf 299} (1984), no.~1, 5--8.


\bibitem[PStV]{PaStV}
I.~Panin, A.~Stavrova, and N.~Vavilov, \emph{On {G}rothendieck-{S}erre's
  conjecture concerning principal {$G$}-bundles over reductive group schemes:
  {I}}, Compos. Math. \textbf{151} (2015), no.~3, 535--567.


\bibitem[PSt1]{PaSt} I.~A. Panin and A.~K. Stavrova, {\it On the {G}rothendieck-{S}erre conjecture
  concerning principal {$G$}-bundles over semilocal {D}edekind domains}, Zap.
  Nauchn. Sem. S.-Petersburg. Otdel. Mat. Inst. Steklov. (POMI) \textbf{443}
  (2016), no.~Voprosy Teorii Predstavleni\u{\i} Algebr i Grupp. 29, 133--146.

\bibitem[PSt2]{PaSt-gth} I. Panin and A. Stavrova, {\it On the Gille theorem for the relative projective line},
arXiv: 2305.16627v3.


\bibitem[PSt3]{PaSt-Har} I. Panin, A. Stavrova, {\it On a theorem of Harder}, arXiv:2502.19223.

\bibitem[Par]{Par82} Raman Parimala, {\it Quadratic spaces over polynomial extensions of regular rings
of dimansion 2}, Math. Ann. {\bf 261} (1982), 287--292.



\bibitem[PeS]{PS}  V. A. Petrov, A.K. Stavrova, {\it Elementary subgroups in isotropic reductive groups},
 Algebra i Analiz {\bf 20} (2008),  160--188; translation in St. Petersburg Math. J. {\bf 20}
 (2009), 625--644.

\bibitem[Pl]{Plo93}
E.~B. Plotkin, \emph{Surjective stabilization of the {$K_1$}-functor for some
  exceptional {C}hevalley groups}, Journal of Soviet Mathematics \textbf{64}
  (1993), no.~1, 751--766.






\bibitem[Q]{Q} D.~Quillen, {\it Projective modules over polynomial rings}, Invent. Math. {\bf 36} (1976),
167--171.

\bibitem[RR]{RaRa}
M.S. Raghunathan, A. Ramanathan, {\it Principal bundles on the
affine line}, Proc.~Indian Acad.~Sci., Math.~Sci. 93 (1984), 137--145.

\bibitem[R89]{Rag} M. S. Raghunathan,
{\it Principal bundles on affine space and bundles on the projective line}, Math. Ann. {\bf 285} (1989), 309--332.




\bibitem[Se02]{Serre-gal}
J.-P. Serre, \emph{Galois cohomology}, english ed., Springer Monographs
  in Mathematics, Springer-Verlag, Berlin, 2002, Translated from the French by
  Patrick Ion and revised by the author.




\bibitem[St19]{St-dh} A. Stavrova, {\it Isotropic reductive groups over discrete Hodge algebras},
J. Homotopy Relat. Str. {\bf 14} (2019), 509--524.

\bibitem[St20]{St-ded} A. Stavrova,
\emph{Chevalley groups of polynomial rings over Dedekind domains}, {\it J. Group Theory} {\bf 23} (2020), 121--132.


 \bibitem[St22]{St-k1}  A.~K.~Stavrova,
 {\it $\mathbb{A}^1$-invariance of non-stable $K_1$-functors in the equicharacteristic case},
Indag. Math. {\bf 33} (2022), 322--333.













\bibitem[Ste]{Ste78}
M.~R. Stein, \emph{Stability theorems for {$K_{1}$}, {$K_{2}$} and related
  functors modeled on {C}hevalley groups}, Japan. J. Math. (N.S.) \textbf{4}
  (1978), no.~1, 77--108.



\bibitem[Su]{Su76} A. Suslin, \emph{Projective modules over polynomial rings are free}, Soviet Math. Dokl.
{\bf 17} (1976), 1160--1164.


\bibitem[SuK]{Sus-K-O1} A.A. Suslin, V.I. Kopeiko, {\it
Quadratic modules and the orthogonal group over polynomial rings},
J. of Soviet Math. {\bf 20} (1982), 2665--2691.






\bibitem[Th]{Thomason}
R.~W. Thomason, {\it Equivariant resolution, linearization, and {H}ilbert's
  fourteenth problem over arbitrary base schemes}, Adv. in Math. \textbf{65}
  (1987), no.~1, 16--34.





\end{thebibliography}
\end{document}